\documentclass{article}[12pt]
\usepackage{amsmath}
\usepackage{amssymb}
\usepackage{amsthm}
\def\C{{\mathbb C}}
\def\N{{\mathbb N}}

\def\R{{\mathbb R}}
\def\K{{\mathbb K}}

\def\H{{\cal H}}
\def\kk{{\cal K}}
\def\epsilon{\varepsilon}
\def\kappa{\varkappa}
\def\phi{\varphi}
\def\leq{\leqslant}
\def\geq{\geqslant}
\def\ssub#1#2{#1_{{}_{{\scriptstyle #2}}}}
\def\supp{{\tt supp}}
\def\cla{{\cal W}}
\def\slim{\mathop{\hbox{$\overline{\hbox{\rm lim}}$}}\limits}

\def\aco{{\tt aconv}}
\def\acon{\overline{\tt aconv}}

\newtheorem{theorem}{Theorem}[section]
\newtheorem{lemma}[theorem]{Lemma}
\newtheorem{corollary}[theorem]{Corollary}
\newtheorem{proposition}[theorem]{Proposition}
\newtheorem{question}[theorem]{Question}
\newtheorem{example}[theorem]{Example}

\numberwithin{equation}{section}

\begin{document}

\title{Norm attaining operators and pseudospectrum}

\author{S. Shkarin}

\date{}

\maketitle

\begin{abstract} \noindent It is shown that if $1<p<\infty$ and
$X$ is a subspace or a quotient of an $\ell_p$-direct sum of finite
dimensional Banach spaces, then for any compact operator $T$ on $X$
such that $\|I+T\|>1$, the operator $I+T$ attains its norm. A
reflexive Banach space $X$ and a bounded rank one operator $T$ on
$X$ are constructed such that $\|I+T\|>1$ and $I+T$ does not attain
its norm.
\end{abstract}

\small \noindent{\bf MSC:} \ \ 

\noindent{\bf Keywords:} \ \ Norm of the resolvent, pseudospectrum,
norm attaining operators \normalsize

\section{Introduction \label{s1}}\rm

All vector spaces in this paper are assumed to be over the field
$\K$, being either the field $\C$ of complex numbers or the field
$\R$ of real numbers. As usual, $\N$ is the set of positive integers
and $\R_+$ is the set of non-negative real numbers. The Banach space
of all bounded linear operators from a Banach space $X$ to a Banach
space $Y$ is denoted by $L(X,Y)$ and $\kk(X,Y)$ stands for the space
of compact linear operators $T:X\to Y$. We write $L(X)$ instead of
$L(X,X)$, $\kk(X)$ instead of $\kk(X,X)$ and $X^*$ instead of
$L(X,\K)$.

We say that $T\in L(X,Y)$ {\it attains its norm on} $x\in X$ if
$\|x\|=1$ and $\|Tx\|=\|T\|$. It is said that $T$ {\it attains its
norm} if there is an $x\in X$ with $\|x\|=1$ such that $T$ attains
its norm on $x$. We would like to mention a few classical results on
the norm attaining property. The James theorem \cite{diest} says
that a Banach space $X$ is reflexive if and only if any $f\in X^*$
attains its norm. As a corollary of the James theorem, we have that
$X$ is reflexive if and only if any $T\in \kk(X)$ attains its norm.
Indeed, if $X$ is non-reflexive, then the James theorem provides a
bounded rank one operator which does not attain its norm. On the
other hand, if $X$ is reflexive and $T\in \kk(X)$, then the function
$x\mapsto \|Tx\|$ is weakly sequentially continuous and the closed
unit ball of $X$ is weakly sequentially compact and therefore the
above function attains its maximum on the unit ball. Below it is
shown that the situation with attaining of the norm for operators
$I+T$ with compact $T$ is quite different. For further results on
the norm attaining operators we refer to
\cite{norm1,norm2,norm3,norm4,norm5,norm6,norm7,norm8,norm9} and
references therein.

The norm attaining property of operators is related to the concept
of the pseudospectrum. Let $X$ be a complex Banach space. For
$\epsilon>0$ and $T\in L(X)$, the $\epsilon$-{\it pseudospectrum} of
$T$ is usually defined as
\begin{equation}\label{ps1}
\sigma_{\epsilon}(T)=\{\lambda\in\C:\|(T-\lambda
I)^{-1}\|>\epsilon^{-1}\}
\end{equation}
or as
\begin{equation}\label{ps2}
\Sigma_{\epsilon}(T)=\{\lambda\in\C:\|(T-\lambda
I)^{-1}\|\geq\epsilon^{-1}\},
\end{equation}
where $\|(T-\lambda I)^{-1}\|$ is assumed to be infinite if
$\lambda$ belongs to the spectrum $\sigma(T)$ of $T$, see, for
instance, \cite{Boet,BGS,CCH,Har,Tref1,Tref2,Sha1,Sha2}. Recently
Shargorodsky \cite{Sha1} demonstrated that the level set
\begin{equation}\label{level}
\Sigma_{\epsilon}(T)\setminus\sigma_{\epsilon}(T)=\{\lambda\in\C:\|(T-\lambda
I)^{-1}\|=\epsilon^{-1}\}
\end{equation}
can have non-empty interior in general, while its interior is empty
when the space $X$ or the dual space $X^*$ is complex uniformly
convex. It is well-known that
\begin{equation}\label{ps3}
\sigma_{\epsilon}(T)=\bigcup_{\|A\|<\epsilon} \sigma(T+A),
\end{equation}
see, for instance, \cite{CCH,GHP}. This equality is one of the main
reasons why many authors prefer (\ref{ps1}) rather than (\ref{ps2})
as the definition of pseudospectrum. We study the question whether
the similar equality holds in the case of non-strict inequalities:
\begin{equation}\label{ps4}
\Sigma_{\epsilon}(T)=\Sigma^0_\epsilon(T),\quad\text{where}\quad
\Sigma^0_\epsilon(T)=\bigcup_{\|A\|\leq\epsilon} \sigma(T+A).
\end{equation}
It is worth noting that the inclusion
\begin{equation}\label{ps5}
\Sigma_\epsilon^0(T)\subseteq \Sigma_{\epsilon}(T)
\end{equation}
holds for any bounded operator $T$ on any Banach space
\cite{CCH,Har}. It is proved by Finck and Ehrhardt, see
\cite{cstar}, that the equality (\ref{ps4}) holds if $X$ is a
Hilbert space. Shargorodsky \cite{Sha2} constructed a bounded linear
operator $T$ on the reflexive space $X=\ell_p\times\ell_q$ with
$1<p<q<\infty$ and the norm $\|(x,y)\|=\|x\|_p+\|y\|_q$ for which
(\ref{ps4}) fails. He also constructed $T\in\kk(\ell_1)$ for which
(\ref{ps4}) fails. These examples naturally lead to the following
question, raised in \cite{Sha2}.

\begin{question}\label{q1} Is it true that $(\ref{ps4})$ holds for any
compact operator on a reflexive complex Banach space?
\end{question}

We show that, in general, the answer to Question~\ref{q1} is
negative and demonstrate that if $1<p<\infty$ and $X$ is an
$\ell_p$-direct sum of finite dimensional Banach spaces, then
(\ref{ps4}) holds for each bounded operator $T$ on $X$. In
particular, it holds when $X=\ell_p$ with $1<p<\infty$. It turns out
that the validity of (\ref{ps4}) for any compact operator $T$ on a
Banach space $X$ is closely related to the norm-attaining property.

\begin{proposition}\label{ll0} Let $X$ be a complex Banach space,
$T\in \kk(X)$, $\epsilon>0$ and $z\in \Sigma_\epsilon(T)$. Then the
following conditions are equivalent:
\begin{itemize}
\item[\rm (\ref{ll0}.1)] $z\in\Sigma^0_\epsilon(T);$
\item[\rm (\ref{ll0}.2)] if
$\|(T-zI)^{-1}\|=\epsilon^{-1}>|z|^{-1}$, then $(T-zI)^{-1}$ attains
its norm.
\end{itemize}
\end{proposition}

We use the above proposition in order to prove the following result.

\begin{proposition}\label{ll} Let $X$ be a complex Banach space.
Then the following conditions are equivalent:
\begin{itemize}
\item[\rm (\ref{ll}.1)] the equality $(\ref{ps4})$ holds for any $\epsilon>0$ and
any $T\in\kk(X);$
\item[\rm (\ref{ll}.2)] for any $T\in\kk(X)$ such
that $I+T$ is invertible and $\|I+T\|>1$, the operator $I+T$ attains
its norm.
\end{itemize}
\end{proposition}

The above proposition motivates the introduction of the following
class of Banach spaces.

\medskip

{\bf Definition 1.} \ We say that a Banach space $X$ belongs to the
class $\cla$ if for each $T\in\kk(X)$ such that $\|I+T\|>1$, $I+T$
attains its norm.

\medskip

From Proposition~\ref{ll} it follows that for any $X\in\cla$ and any
compact operator $T$ on $X$, the equality (\ref{ps4}) holds. It is
also worth noting that the restriction $\|I+T\|>1$ is natural.
Indeed, the diagonal operator $D$ on $\ell_2$ with the diagonal
entries $\{1-2^{-n}\}_{n\in\N}$ has norm 1 which is not attained and
$D$ is the sum of the identity operator and a compact operator. The
following proposition provides a sufficient condition for a Banach
space to belong to $\cla$.

\bigskip

{\bf Definition 2.} \ Let $1<p<\infty$. We say that a Banach space
$X$ is a $p$-space if $X$ is reflexive and for any $x\in X$ and any
weakly convergent to zero sequence $\{u_n\}_{n\in \N}$ in $X$,
\begin{equation} \label{lpclass}
\lim_{n\to\infty}\bigl(\|x+u_n\|-(\|x\|^p+\|u_n\|^p)^{1/p}\bigr)=0.
\end{equation}

\medskip

It is easy to see that any Hilbert space is a $2$-space and that any
finite dimensional Banach space is a $p$-space for any $p$. Note
that an infinite dimensional Banach space cannot be a $p$-space and
a $q$-space for $p\neq q$. Recall that  if $1\leq p<\infty$ and
$\{X_\alpha\}_{\alpha\in\Lambda}$ is a family of Banach spaces, then
their $\ell_p$-direct sum is the space
$$
X=\biggl\{x\in \prod_{\alpha\in\Lambda}
X_\alpha:\sum_{\alpha\in\Lambda} \|x_\alpha\|^p<\infty\biggr\}
$$
endowed with the norm
$$
\|x\|=\biggl(\sum_{\alpha\in\Lambda} \|x_\alpha\|^p\biggr)^{1/p}.
$$
When the family consists of just 2 spaces $X$ and $Y$ we denote its
$\ell_p$-direct sum by $X\oplus_p Y$. We also denote $X\times Y$
with the norm $\|(x,y)\|=\max\{\|x\|,\|y\|\}$ by the symbol
$X\oplus_\infty Y$.

\begin{proposition}\label{uvuv}
Let $1<p<\infty$. Then any closed linear subspace of an
$\ell_p$-direct sum of any family of finite dimensional Banach
spaces is a $p$-space.
\end{proposition}

The following two theorems provide, in particular, a partial
affirmative answer to Question~\ref{q1}. The next one extends the
validity of (\ref{ps4}) for any bounded operator $T$ from just
Hilbert spaces to a wider class of Banach spaces.

\begin{theorem}\label{lp1}
Let $1<p<\infty$ and $X$ be an $\ell_p$-direct sum of a family of
complex finite dimensional Banach spaces. Then $(\ref{ps4})$ holds
for any $T\in L(X)$.
\end{theorem}

In the case of compact operators, we can extend the last theorem.

\begin{theorem}\label{lp2}
Let $1<p\leq q<\infty$, $X$ be a $p$-space and $Y$ be a $q$-space.
Then for any $J\in L(X,Y)$ and $T\in \kk(X,Y)$ such that
$\|J+T\|>\|J\|$, the operator $J+T$ attains its norm. In particular,
any $p$-space belongs to $\cla$.
\end{theorem}

Theorem~\ref{lp2} and Proposition~\ref{ll} imply the following
corollary.

\begin{corollary}\label{lp21}
Let $1<p<\infty$ and $X$ be a complex $p$-space. Then $(\ref{ps4})$
holds for any $T\in\kk(X)$.
\end{corollary}

Even a slight perturbation of the norm destroys the above results.
The following theorem provides a negative answer to
Question~\ref{q1}.

\begin{theorem}\label{simple} Let $1<p<\infty$, $1\leq q\leq\infty$,
$q\neq p$ and $X=\K\oplus_q\ell_p$. Then there exists a compact
operator $T$ on $X$ such that $I+T$ is invertible, $\|I+T\|>1$ and
$I+T$ does not attain its norm.
\end{theorem}

Proposition~\ref{ll} and Theorem~\ref{simple} imply that for
$1<p<\infty$ and $1\leq q\leq\infty$, $q\neq p$ there are
$\epsilon>0$ and $T\in\kk(\C\oplus_q\ell_p)$ such that (\ref{ps4})
fails. Since $\K\oplus_q\ell_p$ is isomorphic to $\ell_p$, we see
that belonging to $\cla$ and validity of (\ref{ps4}) are renorming
sensitive properties. In particular, $\K\oplus_p\ell_2$ is
isomorphic to the Hilbert space $\ell_2$ for any $1\leq p\leq
\infty$ and belongs to $\cla$ if and only if $p=2$. However, for the
spaces $\K\oplus_p\ell_2$ the situation improves if we consider
finite rank operators instead of compact ones.

\begin{proposition}\label{simple1} Let $1\leq p\leq\infty$, $Y$ be a
finite dimensional Banach space and $X=Y\oplus_p\ell_2$. Then for
any bounded finite rank operator $T$ on $X$, the operator $I+T$
attains its norm.
\end{proposition}

The last proposition suggests that the answer to Question~\ref{q1}
might be affirmative if we replace the compactness condition by the
stronger one of $T$ having finite rank. Unfortunately this is not
the case.

\begin{proposition}\label{ex} There exists a norm $\|\cdot\|$ on
$\ell_2$, equivalent to the original norm $\|\cdot\|_2$, and a rank
one operator $T$ on $\ell_2$ such that $T^2=0$, $\|I+T\|=2$ $($with
respect to the norm $\|\cdot\|$ on $\ell_2)$ and the norm of $I+T$
is not attained.
\end{proposition}

The equality $T^2=0$ for the operator from the above proposition
ensures invertibility of $T-I$ and the equality $I+T=-(T-I)^{-1}$.
Thus $(T-I)^{-1}$ does not attain its norm and $\|(T-I)^{-1}\|=2>1$.
By Proposition~\ref{ll0}, $-1\in
\Sigma_{1/2}(T)\setminus\Sigma^0_{1/2}(T)$ and (\ref{ps4}) fails for
$T$ with $\epsilon=1/2$.

\section{Proof of Propositions~\ref{ll0} and~\ref{ll}}

The following lemma is a known fact \cite{CCH,Har}. For convenience
of the reader we reproduce its short proof.

\begin{lemma}\label{att1} Let $X$ be a complex Banach space
$\epsilon>0$ and $T\in L(X)$. Assume also that
$z\in\Sigma_\epsilon(T)\setminus\sigma(T)$ and $(T-zI)^{-1}$ attains
its norm. Then there is $A\in L(X)$ such that $\|A\|\leq\epsilon$
and $z\in\sigma(T+A)$.
\end{lemma}

\begin{proof} Since $(T-zI)^{-1}$ attains its norm, there exist
$x,y\in X$ such that $\|y\|=\|x\|=1$ and $(T-zI)^{-1}x=cy$, where
$c=\|(T-zI)^{-1}\|$. Using the Hahn--Banach theorem, we can pick
$f\in X^*$ for which $\|f\|=f(y)=1$. Consider the operator
$$
A\in L(X),\qquad Au=-c^{-1}f(u)x.
$$
Clearly $\|A\|\leq c^{-1}$. Since $z\in\Sigma_\epsilon(T)$, we have
$c^{-1}\leq \epsilon$. Thus $\|A\|\leq\epsilon$. Moreover,
$Ay=-c^{-1}x$. From the equality $(T-zI)^{-1}x=cy$ it follows that
$Ty=zy+c^{-1}x$. Hence $(T+A)y=zy$ and $z\in\sigma(T+A)$.
\end{proof}

\subsection{Proof of Proposition~\ref{ll0}}

If $X$ is finite dimensional, then any $S\in L(X)$ attains its norm
and according to Lemma~\ref{att1}, both (\ref{ll0}.1) and
(\ref{ll0}.2) are satisfied. Thus for the rest of the proof, we can
assume that $X$ is infinite dimensional.

Assume that (\ref{ll0}.2) is satisfied. Since $X$ is infinite
dimensional and $T$ is compact, $\|(T-zI)^{-1}\|\geq |z|^{-1}$. If
the relation $\|(T-zI)^{-1}\|=\epsilon^{-1}>|z|^{-1}$ fails, then
either $\|(T-zI)^{-1}\|>\epsilon^{-1}$ or
$\|(T-zI)^{-1}\|=\epsilon^{-1}=|z|^{-1}$. If
$\|(T-zI)^{-1}\|>\epsilon^{-1}$, then $z\in\sigma_\epsilon(T)$ and,
according to (\ref{ps3}), $z\in\Sigma^0_\epsilon(T)$. If
$\|(T-zI)^{-1}\|=\epsilon^{-1}=|z|^{-1}$, then $\|A\|=\epsilon$,
where $A=zI$. Since $X$ is infinite dimensional and $T$ is compact,
we have $0\in\sigma(T)$. Hence $z\in \sigma(T+zI)=\sigma(T+A)$. Thus
$z\in\Sigma^0_\epsilon(T)$. It remains to consider the case when
$\|(T-zI)^{-1}\|=\epsilon^{-1}>|z|^{-1}$ and $(T-zI)^{-1}$ attains
its norm. In this case, from Lemma~\ref{att1} it follows that $z\in
\Sigma^0_\epsilon(T)$. The implication
$(\ref{ll0}.2)\Longrightarrow(\ref{ll0}.1)$ is verified.

Assume now that (\ref{ll0}.1) is satisfied. That is, there exists
$A\in L(X)$ such that $\|A\|\leq \epsilon$ and $z\in\sigma(T+A)$.
Hence $0\in \sigma(T-zI+A)$. Suppose that (\ref{ll0}.2) fails. Then
$\|(T-zI)^{-1}\|=\epsilon^{-1}>|z|^{-1}$ and the norm of
$(T-zI)^{-1}$ is not attained. Since $\|A\|\leq\epsilon$ and
$|z|>\epsilon$, the operator $-zI+A$ is invertible. Then $T-zI+A$ is
a Fredholm operator of index zero as a sum of a compact operator $T$
and an invertible operator $-zI+A$. Since $0\in \sigma(T-zI+A)$,
$T-zI+A$ is non-invertible and therefore, being a Fredholm operator
of index zero, it has non-trivial kernel. Thus we can pick $x\in X$
such that $\|x\|=1$ and $(T-zI+A)x=0$. It follows that $-Ax=(T-zI)x$
and therefore $x=-(T-zI)^{-1}Ax$. Using the relations
$\|(T-zI)^{-1}\|=\epsilon^{-1}$ and $\|A\|\leq\epsilon$, we obtain
$$
1=\|x\|=\|-(T-zI)^{-1}Ax\|\leq \epsilon^{-1}\|Ax\|\leq
\epsilon^{-1}\|A\|\|x\|\leq \|x\|=1.
$$
Obviously, all inequalities in the above display should be
equalities which can only happen if $\|A\|=\|Ax\|=\epsilon$. Then
$\|y\|=1$, where $y=-\epsilon^{-1}Ax$. Since $-(T-zI)^{-1}Ax=x$, we
obtain $(T-zI)^{-1}y=\epsilon^{-1}x$. Thus
$\|(T-zI)^{-1}y\|=\epsilon^{-1}=\|(T-zI)^{-1}\|$. That is,
$(T-zI)^{-1}$ attains its norm on $y$. This contradiction completes
the proof of the implication
$(\ref{ll0}.1)\Longrightarrow(\ref{ll0}.2)$ and that of
Proposition~\ref{ll0}.

\subsection{Proof of Proposition~\ref{ll}}

First, assume that (\ref{ll}.2) is satisfied. Let also $T\in\kk(X)$,
$\epsilon>0$ and $z\in \Sigma_\epsilon(T)$. According to
(\ref{ps5}), it suffices to show that $z\in\Sigma^0_\epsilon(T)$. By
Proposition~\ref{ll0}, the latter happens if and only if
(\ref{ll0}.2) is satisfied. Assume that it is not the case. Then
$\|(T-zI)^{-1}\|=\epsilon^{-1}>|z|^{-1}$ and the norm of
$(T-zI)^{-1}$ is not attained. On the other hand,
$(T-zI)^{-1}=-z^{-1}(I+S)$, where $S=-z(T-zI)^{-1}-I$ is compact.
Moreover $\|I+S\|>1$ since $\|(T-zI)^{-1}\|>|z|^{-1}$. By
(\ref{ll}.2), $I+S$ attains its norm and therefore so does
$(T-zI)^{-1}$. This contradiction proves the implication
$(\ref{ll}.2)\Longrightarrow(\ref{ll}.1)$.

Next, assume that (\ref{ll}.1) is satisfied, $T\in\kk(X)$, $I+T$ is
invertible and $c=\|I+T\|>1$. Let $S=(I+T)^{-1}-I$. Clearly $S$ is
compact and $I+T=(S+I)^{-1}$. Let $\epsilon=c^{-1}$. Since
$\|(S+I)^{-1}\|=\|I+T\|=\epsilon^{-1}>1$, we have
$-1\in\Sigma_\epsilon(S)$. According to (\ref{ll}.2), $-1\in
\Sigma_\epsilon^0(S)$. Proposition~\ref{ll0} implies now that
$(S+I)^{-1}=I+T$ attains its norm. This completes the proof of the
implication $(\ref{ll}.1)\Longrightarrow(\ref{ll}.2)$ and that of
Proposition~\ref{ll}.

\section{$\ell_p$-direct sums of finite dimensional Banach spaces}

Throughout this section $1<p<\infty$ and $X$ is the $\ell_p$-direct
sum of a family $\{X_\alpha:\alpha\in\Lambda\}$ of finite
dimensional Banach spaces. For $x\in X$, the support of $x$ is the
set
$$
\supp(x)=\{\alpha\in\Lambda:x_\alpha\neq 0\}.
$$
From the definition of the $\ell_p$-direct sum it follows that the
support of any element of $X$ is at most countable. For a subset $B$
of $\Lambda$, we consider $P_B\in L(X)$ defined by the formula
\begin{equation}\label{pb}
(P_Bx)_\alpha=\left\{\begin{array}{ll}x_\alpha&\text{if $\alpha\in
B$,}\\0&\text{if $\alpha\notin B$.}\end{array} \right.
\end{equation}
Clearly $P_B$ is a linear projection and $\|P_B\|=\|I-P_B\|=1$ if
$B$ is non-empty and $B\neq\Lambda$.

\begin{lemma} \label{weak1} Let $\{x_n\}_{n\in\N}$ be a sequence
in $X$ weakly convergent to zero and $\{\epsilon_k\}_{k\in\N}$ be a
sequence of positive numbers. Then there exist a strictly increasing
sequence $\{n_k\}_{k\in\N}$ of positive numbers and a sequence
$\{u_k\}_{k\in\N}$ of elements of $X$ such that
$\|x_{n_k}-u_k\|<\epsilon_k$ for each $k\in\N$ and the sets $\supp
(u_k)$ are finite and pairwise disjoint.
\end{lemma}

\begin{proof} We construct the required sequences inductively. On
the first step we take $n_1=1$, pick a finite subset $B$ of
$\Lambda$ such that $\|x_1-P_{B}x_1\|<\epsilon_1$ and put
$u_1=P_{B}x_1$.

Assume now that $k\geq 2$, $n_1<{\dots}<n_{k-1}$,
$u_1,\dots,u_{k-1}$ are vectors in $X$ with pairwise disjoint finite
supports such that $\|x_{n_j}-u_j\|<\epsilon_j$ for $1\leq j\leq
k-1$. Let now $C$ be the union of $\supp (u_j)$ for $1\leq j\leq
k-1$. Since $C$ is finite, $P_C$ is a compact operator and therefore
$\|P_C x_n\|\to 0$ as $n\to\infty$. Thus we can pick $n_k>n_{k-1}$
such that $\|P_C x_{n_k}\|<\epsilon_k/2$. Next, choose a finite
subset $A$ of $\Lambda$ such that $C\subseteq A$ and
$\|x_{n_k}-P_Ax_{n_k}\|<\epsilon_k/2$ and put $u_k=P_Ax_{n_k}-P_C
x_{n_k}$. Clearly $\supp(u_k)\subseteq A\setminus C$ and therefore
$u_k$ has finite support and the supports of $u_1,\dots,u_k$ are
pairwise disjoint. Finally,
$$
\|x_{n_k}-u_k\|= \|P_Cx_{n_k}+(x_{n_k}-P_Ax_{n_k})\|\leq
\|x_{n_k}-P_Ax_{n_k}\|+\|P_C x_{n_k}\|<
\epsilon_k/2+\epsilon_k/2=\epsilon_k.
$$

The description of the inductive construction of sequences $\{n_k\}$
and $\{u_k\}$ is now complete.
\end{proof}

\begin{lemma}\label{weak2} $X$ is a $p$-space.
\end{lemma}

\begin{proof} Since $1<p<\infty$, $X$ is reflexive as an
$\ell_p$-direct sum of reflexive Banach spaces. Let $x\in X$ and
$\{u_n\}_{n\in\N}$ be a sequence in $X$ weakly convergent to $0$.
Let also $\epsilon>0$. Pick a finite subset $B$ of $\Lambda$ such
that $\|x-P_Bx\|<\epsilon$. Since $P_B$ is a compact operator and
$\{u_n\}$ converges weakly to 0, we have $\|P_Bu_n\|\to 0$ as
$n\to\infty$. Let $x_n=u_n-P_Bu_n$. Then $\|x_n-u_n\|\to 0$ as
$k\to\infty$ and supports of $x_n$ do not meet $B$. Since the
support of $P_Bx$ is contained in $B$, the supports of $x_n$ do not
intersect the support of $P_Bx$ and from the definition of the norm
on $X$ it follows that
$$
\|P_Bx+x_n\|=(\|P_Bx\|^p+\|x_n\|^p)^{1/p}.
$$
Since $\|x-P_Bx\|<\epsilon$ and $\|x_n-u_n\|\to 0$ as $n\to\infty$,
we see that
$$
|\|P_Bx+x_n\|-\|x+u_n\||<\epsilon\ \ \text{and}\ \
\bigl|(\|x\|^p+\|u_n\|^p)^{1/p}-(\|P_Bx\|^p+\|x_n\|^p)^{1/p}\bigr|<\epsilon
$$
for all sufficiently large $n$. From the last two displays it
follows that
$$
\|x+u_n\|-(\|x\|^p+\|u_n\|^p)^{1/p}<2\epsilon
$$
for all sufficiently large $n$. Since $\epsilon>0$ is arbitrary, the
equality (\ref{lpclass}) follows.
\end{proof}

\subsection{Proof of Proposition~\ref{uvuv}}

By Lemma~\ref{weak2}, the class of $p$-spaces contains
$\ell_p$-direct sums of finite dimensional Banach spaces. From the
definition it immediately follows that (closed linear) subspaces of
$p$-spaces are $p$-spaces. Hence closed linear subspaces of
$\ell_p$-direct sums of finite dimensional Banach spaces are
$p$-spaces.

\subsection{Operators on $p$-spaces}

\begin{lemma}\label{lp2222} Let $1<p\leq q<\infty$, $Y$ be a
$p$-space, $Z$ be a $q$-space, $T\in L(Y,Z)$, $\{x_n\}_{n\in\N}$ be
a sequence in $Y$ such that $\|x_n\|\to 1$ and $\|Tx_n\|\to \|T\|$
as $n\to\infty$ and $\{x_n\}$ is weakly convergent to $x\in X$. Then
$\|Tx\|=\|T\|\|x\|$.
\end{lemma}

\begin{proof} Let $u_n=x_n-x$. Then $\{u_n\}$ is weakly convergent
to $0$. Since $T$ is linear and bounded, $T$ is also continuous with
respect to the weak topology and therefore $\{Tu_n\}$ is weakly
convergent to 0. For brevity denote $c=\|T\|$. Since $Y$ is a
$p$-space and $Z$ is a $q$-space, we have
\begin{align}\label{a1}
1&=\lim_{n\to\infty}\|x_n\|=\lim_{n\to\infty}\|x+u_n\|=
\lim_{n\to\infty}(\|x\|^p+\|u_n\|^p)^{1/p},
\\
\label{a2}
c&=\lim_{n\to\infty}\|Tx_n\|=\lim_{n\to\infty}\|Tx+Tu_n\|=
\lim_{n\to\infty}(\|Tx\|^q+\|Tu_n\|^q)^{1/q}.
\end{align}

Clearly $\|Tu_n\|\leq c\|u_n\|$ for each $n\in\N$. Assume that
$\|Tx\|\neq c\|x\|$. Then $\|Tx\|<c\|x\|$. Using these inequalities
together with (\ref{a1}) and (\ref{a2}) and taking into account that
$p\leq q$, we obtain
\begin{equation*}
c=\lim_{n\to\infty}(\|Tx\|^q+\|Tu_n\|^q)^{1/q}<
\lim_{n\to\infty}(c^q\|x\|^q+c^q\|u_n\|^q)^{1/q}\leq c
\lim_{n\to\infty}(\|x\|^p+\|u_n\|^p)^{1/p}=c.
\end{equation*}
This contradiction proves the equality $\|Tx\|=c\|x\|$.
\end{proof}

Recall that $X$ is the $\ell_p$-direct sum of the family
$\{X_\alpha:\alpha\in \Lambda\}$ of finite dimensional Banach
spaces.

\begin{lemma}\label{lp111} Let $T\in L(X)$ be
such that
\begin{equation}\label{inf0}
\inf_{\|x\|=1}\|Tx\|=c>0.
\end{equation}
Then there exists $S\in L(X)$ such that $\|S\|=c$ and
\begin{equation}\label{inf1}
\inf_{\|x\|=1}\|(T+S)x\|=0.
\end{equation}
\end{lemma}

\begin{proof} Pick a sequence $\{x_n\}_{n\in \N}$ in $X$ such that
$\|x_n\|\to 1$ and $\|Tx_n\|\to c$ as $n\to\infty$. Since $X$ is
reflexive, we can choose such a sequence $\{x_n\}$ being weakly
convergent to $x\in X$. Clearly $\|x\|\leq 1$.

{\bf Case} \ $x=0$. That is, $\{x_n\}$ weakly converges to $0$. By
Lemma~\ref{weak1}, we can find a strictly increasing sequence
$\{n_k\}_{k\in\N}$ of positive integers and a sequence
$\{y_k\}_{k\in\N}$ of elements of $X_0$ such that the supports of
$y_k$ are pairwise disjoint and
\begin{equation}\label{yk}
\text{$\|x_{n_k}-y_k\|<2^{-k}$ for any $k\in\N$.}
\end{equation}
Since the sequence $\{x_{n_k}\}$ weakly converges to 0, formula
(\ref{yk}) implies that $\{y_k\}$ also weakly converges to 0. Since
$T\in L(X)$, the sequence $\{Ty_k\}$ weakly converges to 0. Using
Lemma~\ref{weak1} once again, we see that there exist a strictly
increasing sequence $\{k_m\}_{m\in\N}$ of positive integers and a
sequence $\{w_m\}_{m\in\N}$ in $X_0$ such that the supports of $w_m$
are pairwise disjoint and
\begin{equation}\label{wm}
\text{$\|Ty_{k_m}-w_m\|<2^{-m}$ for any $m\in\N$.}
\end{equation}
From (\ref{yk}) it follows that $\|Tx_{n_k}-Ty_k\|\leq 2^{-k}\|T\|$
for any $k\in\N$. Since $\|Tx_{n_k}\|\to c$ as $k\to\infty$, we have
$\|Ty_k\|\to c$ as $k\to\infty$. Now by (\ref{wm}) and (\ref{yk}) we
obtain
\begin{equation}\label{lino}
\lim_{m\to\infty}\|y_{k_m}\|=1\ \ \ \text{and}\ \ \
\lim_{m\to\infty}\|w_m\|=c.
\end{equation}

For each $m\in\N$ let $A_m=\supp(y_{k_m})$, $P_m=P_{A_m}$ and
$X_m=P_m(X)$. By Hahn--Banach theorem, for any $m\in\N$, we can find
$\phi_m\in X_m^*$ such that $\|\phi_m\|=1$ and
$\phi_m(y_{k_m})=\|y_{k_m}\|$. Consider the operator $S\in L(X)$
defined by the formula
$$
Su=-c\sum_{m=1}^\infty \frac{\phi_m(P_mu)}{\|w_m\|}w_m.
$$
From the equalities $\|\phi_m\|=1$, pairwise disjointness of $A_m$
and pairwise disjointness of $\supp(w_m)$ it immediately follows
that $\|S\|=c$. On the other hand, by the definition of $S$
$$
Sy_{k_m}=-\frac{c\|y_{k_m}\|}{\|w_m\|}w_m\ \ \ \text{for any
$m\in\N$}.
$$
According to (\ref{lino}) we have $\|Sy_{k_m}+w_m\|\to 0$ as
$m\to\infty$. Thus by (\ref{wm}), $\|(T+S)y_{k_m}\|\to 0$ as
$m\to\infty$. From (\ref{lino}) it follows that $\|y_{k_m}\|\to 1$
as $m\to\infty$. Hence (\ref{inf1}) is satisfied.

{\bf Case} $x\neq 0$. Let $Y=T(X)$. According to (\ref{inf0}), $Y$
is a closed linear subspace of $X$ and $T:X\to Y$ is invertible.
Consider $R\in L(Y,X)$ being the inverse of $T:X\to Y$. From
(\ref{inf0}) it follows that $\|R\|=c^{-1}$. It is also clear that
the sequence $u_n=c^{-1}Tx_n$ is weakly convergent to $c^{-1}Tx$ and
$\|u_n\|\to 1$ as $n\to\infty$. Moreover $Ru_n=c^{-1}x_n$ for any
$n\in\N$ and therefore $Ru_n$ weakly converges to $c^{-1}x$ and
$\|Ru_n\|\to c^{-1}=\|R\|$  as $n\to\infty$. By
Proposition~\ref{uvuv}, $X$ and $Y$ are $p$-spaces. Hence, according
to Lemma~\ref{lp2222}, $\|R\|\|c^{-1}Tx\|=\|c^{-1}RTx\|$. Taking
into account that $RTx=x$ and $\|R\|=c^{-1}$, we have
$\|Tx\|=c\|x\|$. By Hahn--Banach theorem, we can find $\phi\in X^*$
such that $\|\phi\|=1$ and $\phi(x)=\|x\|$. Let now $S\in L(X)$,
$Su=-\|x\|^{-1}\phi(u)Tx$. Since $\|Tx\|=c\|x\|$, we have $\|S\|\leq
c$. Moreover $(T+S)x=Tx-Tx=0$ and therefore $T+S$ has non-trivial
kernel. Hence (\ref{inf1}) is satisfied.
\end{proof}

\subsection{Proof of Theorem~\ref{lp1}}

Let $T\in L(X)$, $\epsilon>0$ and $z\in\Sigma_\epsilon(T)$. In view
of (\ref{ps5}), it suffices to show that $z\in\Sigma^0_\epsilon(T)$.
Since $z\in\Sigma_\epsilon(T)$, we have $\|(T-zI)^{-1}\|\geq
\epsilon^{-1}$. If $\|(T-zI)^{-1}\|>\epsilon^{-1}$, the inclusion
$z\in\Sigma^0_\epsilon(T)$ follows from (\ref{ps3}). It remains to
consider the case $\|(T-zI)^{-1}\|=\epsilon^{-1}$. In this case
$$
\epsilon=\inf_{\|x\|=1}\|(T-zI)x\|.
$$
By Lemma~\ref{lp111}, we can find $S\in L(X)$ such that
$\|S\|\leq\epsilon$ and
$$
0=\inf_{\|x\|=1}\|(T-zI+S)x\|.
$$
The last display implies that $T+S-zI$ is not invertible. Hence
$z\in\sigma(T+S)$. Since $\|S\|\leq \epsilon$, we obtain the
required inclusion $z\in\Sigma^0_\epsilon(T)$.

\subsection{Proof of Theorem~\ref{lp2}}

\begin{lemma}\label{xnz} Let $Y$ and $Z$ be Banach spaces, $J\in
L(Y,Z)$ and $T\in\kk(Y,Z)$ be such that $\|J+T\|>\|J\|$. Assume also
that $\{x_n\}_{n\in\N}$ is a sequence of vectors in $Y$ weakly
convergent to $x\in Y$ for which $\|x_n\|\to 1$ and
$\|(J+T)x_n\|\to\|J+T\|$ as $n\to\infty$. Then $x\neq 0$.
\end{lemma}

\begin{proof} Since $T$ is compact, $\|Tx_n-Tx\|\to 0$ as
$n\to\infty$. Hence $\lim\limits_{n\to\infty}\|Jx_n+Tx\|=\|J+T\|$.
On the other hand, $\slim\limits_{n\to\infty}\|Jx_n\|\leq \|J\|$.
Thus using the triangle inequality, we obtain
$\|Tx\|\geq\|J+T\|-\|J\|$. It follows that $\|x\|\|T\|\geq
\|Tx\|\geq \|J+T\|-\|J\|>0$. Hence, $x\neq 0$. \end{proof}

We are ready to prove Theorem~\ref{lp2}. Pick a sequence
$\{x_n\}_{n\in\N}$ of elements of $X$ such that $\|x_n\|=1$ for any
$n\in\N$ and $\|(J+T)x_n\|\to \|J+T\|$ as $n\to\infty$. Since $X$ is
reflexive, we, passing to a subsequence, if necessary, may assume
that $\{x_n\}$ weakly converges to $x\in X$. By Lemma~\ref{xnz}
$x\neq 0$. According to Lemma~\ref{lp2222},
$\|Jx+Tx\|=\|x\|\|J+T\|$. Hence $J+T$ attains its norm on the vector
$x/\|x\|$.

\section{Operators on $\K\oplus_q\ell_p$}

We start by a series of elementary observations.

\begin{lemma} \label{du1} Let $X$ and $Y$ be Banach spaces and $T\in
L(X,Y)$ be an operator attaining its norm. Then the dual operator
$T^*\in L(Y^*,X^*)$ attains its norm.
\end{lemma}

\begin{proof} Since $T$ attains its norm, there exists $x\in X$
such that $\|x\|=1$ and $\|Tx\|=\|T\|$. By Hahn--Banach theorem, we
can pick $\phi\in Y^*$ such that $\|\phi\|=1$ and $\phi(Tx)=\|Tx\|$.
Since $\phi(Tx)=(T^*\phi)(x)$, we have $\phi(Tx)\leq
\|T^*\phi\|\|x\|=\|T^*\phi\|$. Since $\|Tx\|=\|T\|=\|T^*\|$, we see
that $\|T^*\phi\|\geq \|T^*\|$ and $\|\phi\|=1$. Thus
$\|T^*\phi\|=\|T^*\|$ and therefore $T^*$ attains its norm at
$\phi$.
\end{proof}

The above lemma immediately implies the following corollary.

\begin{corollary}\label{du2} Let $X$ be a reflexive Banach space and
$T\in L(X)$. Then $T$ attains its norm if and only if $T^*$ attains
its norm.
\end{corollary}

In particular, using the facts that an operator is compact if and
only if its dual is compact and an operator is invertible if and
only if its dual is invertible, we have the following result.

\begin{corollary}\label{du3} Let $X$ be a reflexive Banach space. Then $X\in\cla$
if  and only if $X^*\in \cla$. Moreover, the following two
statements are equivalent:
\begin{itemize}
\item[\rm(\ref{du3}.1)] there is a compact operator $T$ on $X$ such
that $I+T$ is invertible, $\|I+T\|>1$ and $I+T$ does not attain its
norm;
\item[\rm(\ref{du3}.2)] there is a compact operator $S$ on $X^*$ such
that $I+S$ is invertible, $\|I+S\|>1$ and $I+S$ does not attain its
norm.
\end{itemize}
\end{corollary}

\subsection{Proof of Theorem~\ref{simple}}

Let $1<p<\infty$, $1\leq q\leq\infty$, $p\neq q$ and
$X=\K\oplus_q\ell_p$. Clearly $X$ is reflexive and $X^*$ is
naturally isometrically isomorphic to $\K\oplus_{q'}\ell_{p'}$,
where $\frac1p+\frac1{p'}=\frac1q+\frac1{q'}=1$. It is also easy to
see that $p'>q'$ if $p<q$ and $p'<q'$ if $p>q$. According to
Corollary~\ref{du3}, it is enough to prove Theorem~\ref{simple} in
the case $p<q$. Thus from now on, we assume that $p<q$.

We naturally interpret $X$ as a space of sequences $x=\{x_n\}_{n\geq
0}$, where $x_0$ and $\{x_n\}_{n\in\N}$ correspond to the
$\K$-component and the $\ell_p$-component in the decomposition
$X=\K\oplus_q\ell_p$ respectively. For any $x\in X$, we denote
\begin{equation}\label{abg}
\alpha(x)=|x_0|,\quad \beta(x)=|x_1| \quad\text{and}\quad
\gamma(x)=\biggl(\sum_{n=2}^\infty|x_n|^p\biggr)^{1/p}.
\end{equation}
Clearly, for $x\in X$,
\begin{align}
\|x\|&=f(\alpha(x),\beta(x),\gamma(x)),\ \ \text{where} \label{no1}
\\
f(\alpha,\beta,\gamma)&=\left\{\begin{array}{ll}
\bigl(\alpha^q+(\beta^p+\gamma^p)^{q/p}\bigr)^{1/q}&\text{if
$q<\infty$,}\\
\max\bigl\{\alpha,(\beta^p+\gamma^p)^{1/p}\bigr\}&\text{if
$q=\infty$.}\end{array}\right. \label{no2}
\end{align}
Consider the operator $S\in L(X)$ defined by the formula
$(Sx)_0=x_1$, $(Sx)_1=x_0$ and $(Sx)_n=\frac{nx_n}{n+1}$ if $n\geq
2$. That is,
$$
Sx=\Bigl(x_1,x_0,\frac{2x_2}3,\frac{3x_3}{4},\dots\Bigr).
$$
Clearly $T=S-I$ is compact. Thus in order to verify that $T$
satisfies the required conditions, it suffices to show that $S=I+T$
is invertible, $\|S\|>1$ and $S$ does not attain its norm.
Invertibility of $S$ is obvious. Indeed, the operator $R\in L(X)$
defined as $Rx=\bigl(x_1,x_0,3x_2/2,4x_3/3,\dots\bigr)$ is the
inverse of $S$. Next, let $x\in X$ be such that $\|x\|=1$ and let
$\alpha(x)$, $\beta(x)$ and $\gamma(x)$ be the numbers defined in
(\ref{abg}). It is clear that $\alpha(Sx)=\beta(x)$,
$\beta(Sx)=\alpha(x)$, $\gamma(Sx)\leq\gamma(x)$. Moreover,
$\gamma(Sx)<\gamma(x)$ if $\gamma(x)>0$. Thus according to
(\ref{no1}),
\begin{equation}\label{C1}
f(\alpha(x),\beta(x),\gamma(x))=1\ \ \text{and}\ \
\|Sx\|=f(\beta(x),\alpha(x),\gamma(Sx))\leq
f(\beta(x),\alpha(x),\gamma(x)).
\end{equation}
Moreover, since $\gamma(Sx)<\gamma(x)$ when $\gamma(x)>0$, we have
\begin{equation}\label{C2}
\begin{array}{rl}\|Sx\|<f(\beta(x),\alpha(x),\gamma(x))&\text{if $q<\infty$
and $\gamma(x)>0$}\\ \text{and}&\text{if $q=\infty$, $\gamma(x)>0$
and $\beta(x)<(\alpha(x)^p+\gamma(x)^p)^{1/p}$}.\end{array}
\end{equation}
According to (\ref{C1}), $\|Sx\|\leq C$, where
$$
C=\sup\{f(\beta,\alpha,\gamma):(\alpha,\beta,\gamma)\in K\}\ \
\text{and}\ \ K=\{(\alpha,\beta,\gamma)\in\R^3_+:\
f(\alpha,\beta,\gamma)=1\}.
$$
Since $K$ is compact and $f$ is continuous, the supremum in the
definition of $C$ is attained. Using, for instance, the Lagrange
multipliers technique, one can easily see that the function
$(\alpha,\beta,\gamma)\mapsto f(\beta,\alpha,\gamma)$ from $K$ to
$\R_+$ attains its maximal value $C=2^{\frac1p-\frac1q}$ in exactly
one point being $(2^{-1/q},0,2^{-1/q})$. From (\ref{C2}) it now
follows that
\begin{equation}\label{C3}
\|Sx\|<C=2^{\frac1p-\frac1q}\ \ \text{whenever $\|x\|=1$.}
\end{equation}
Now consider the sequence $x_n=2^{-1/q}e_0+2^{-1/q}e_n$, $n\in\N$,
where $\{e_k\}$  is the canonical basis in the sequence space $X$.
Clearly $\|x_n\|=1$ for each $n\in\N$. On the other hand, for any
$n\geq 2$, $Sx_n=2^{-1/q}\Bigl(e_1+\frac{n}{n+1}e_n\Bigr)$ and
therefore
\begin{equation}\label{C4}
\|Sx_n\|=2^{-1/q}\biggl(1+\Bigl(\frac{n}{n+1}\Bigr)^p\biggr)^{1/p}\to
2^{\frac1p-\frac1q}\ \ \text{as}\ \ n\to\infty.
\end{equation}
From (\ref{C3}) and (\ref{C4}) it follows that
$\|S\|=2^{\frac1p-\frac1q}>1$ and the norm of $S$ is not attained.
The proof of Theorem~\ref{simple} is now complete.

\section{Proper extensions of Hilbert spaces and finite rank
operators}

In this section we prove a theorem slightly stronger then
Proposition~\ref{simple1}. We need some preparation. Throughout this
section $\H$ is a Hilbert space and $n\in\N$. We say that
$X=\K^n\times \H$ is a {\it proper extension} of $\H$ if $X$ is
endowed with a norm such that
\begin{equation}\label{phi}
\|(t,x)\|=\phi(t,\|x\|)\quad\text{for any $t\in\K^n$, $x\in\H$},
\end{equation}
where $\phi:\K^n\times\R_+\to\R_+$ is a function and $\phi(0,1)=1$.
The fact that $(t,x)\mapsto\phi(t,\|x\|)$ is a norm on $X$ implies
immediately that $\phi$ is Lipschitzian, convex, $\phi(t,a)>0$,
whenever $(t,a)\neq (0,0)$ and $\phi(st,sa)=s\phi(t,a)$ for any
$s,a\in\R_+$ and $t\in\K^n$. The normalization condition
$\phi(0,1)=1$ implies that $\|(0,x)\|=\ssub{\|x\|}{\H}$ for any
$x\in \H$. Thus $\H$ is naturally isometrically embedded into $X$.
Since $\H$ has finite codimension in $X$, we see that $X$ is a
Banach space and admits an equivalent norm which turns it into a
Hilbert space. In particular, $X$ is reflexive.

\begin{theorem}\label{simple2} Let $X=\K^n\times \H$ be a proper
extension of a Hilbert space $\H$. Then for any bounded finite rank
operator $T$ on $X$, $I+T$ attains its norm.
\end{theorem}

\begin{proof} Let $\phi:\K^n\times\R_+\to\R_+$ be a function
defining the norm on $X$ according to (\ref{phi}). If $\H$ is finite
dimensional, the result becomes trivial. Thus we can assume that
$\H$ is infinite dimensional. Pick a sequence
$\{\xi_k=(t_k,x_k)\}_{k\in\N}$ of elements of $X$ such that
$\|\xi_k\|\to 1$ and $\|(I+T)\xi_k\|\to c$ as $k\to\infty$. Since
$X$ is reflexive, we can, passing to a subsequence, if necessary,
assume that $\{\xi_k\}$ converges weakly to $\xi=(t,x)\in X$. Since
$T$ has finite rank, $\{T\xi_k\}$ is norm convergent to $T\xi=(s,y
)$. Next, since weak and norm convergences on a finite dimensional
Banach space coincide, we see that $\{t_k\}$ converges to $t$ in
$\C^n$. Passing to a subsequence again, if necessary, we can assume
that $\|x_k-x\|\to \alpha\in\R_+$.

Since $x_k-x$ is weakly convergent to zero in the Hilbert space $\H$
and any Hilbert space is a 2-space, we see that
\begin{equation}\label{012}
\lim_{k\to\infty}\|x_k\|=(\|x\|^2+\alpha^2)^{1/2}\ \ \ \text{and}\ \
\ \lim_{k\to\infty}\|x_k+y\|=(\|x+y\|^2+\alpha^2)^{1/2}.
\end{equation}
Since $\|\xi_k\|\to 1$, $\|(I+T)\xi_k\|\to c$, $t_k\to t$ and
$\|T\xi_k-(s,y)\|\to 0$, we have
\begin{align*}
1&=\lim_{k\to\infty}\|\xi_k\|=\lim_{k\to\infty}\|(t_k,x_k)\|=\lim_{k\to\infty}\|(t,x_k)\|,
\\
c&=\lim_{k\to\infty}\|(I+T)\xi_k\|=\lim_{k\to\infty}\|\xi_k+(s,y)\|=
\lim_{k\to\infty}\|(t_k+s,x_k+y)\|=\lim_{k\to\infty}\|(t+s,x_k+y)\|.
\end{align*}
Using (\ref{phi}), (\ref{012}) and continuity of $\phi$, we obtain
\begin{equation}\label{eq1}
\phi(t,(\|x\|^2+\alpha^2)^{1/2})=1\ \ \text{and}\ \
\phi(t+s,(\|x+y\|^2+\alpha^2)^{1/2})=c.
\end{equation}

Since $\H$ is infinite dimensional and $T$ has finite rank, the
linear subspace
$$
L=\{v\in H:T(0,v)=0,\ \langle u,x\rangle=\langle u,y\rangle=0\}
$$
has finite codimension and therefore is non-trivial. Hence we can
pick $u\in L$ such that $\|u\|=\alpha$. Since $u$ is orthogonal to
both $x$ and $y$, we see that $\|x+u\|=(\|x\|^2+\alpha^2)^{1/2}$ and
$\|x+y+u\|=(\|x+y\|^2+\alpha^2)^{1/2}$. Hence, according to
(\ref{phi}) and (\ref{eq1})
$$
\|(t,x+u)\|=\phi(t,\|x+u\|)=1\ \ \ \text{and}\ \ \
\|(t+s,x+y+u)\|=\phi(t,\|x+y+u\|)=c.
$$
Finally, since $T(0,u)=0$, we have $T(t,x+u)=T(t,x)=(s,y)$. Hence
$$
(I+T)(t,x+u)=(t+s,x+y+u).
$$
Since $c=\|I+T\|$, from the last two displays it follows that $I+T$
attains its norm on the vector $(t,x+u)$.
\end{proof}

Theorem~\ref{simple1} follows from Theorem~\ref{simple2} since
$Y\oplus_p\ell_2$ for a finite dimensional Banach space $Y$ is a
particular case of a proper extension.

\section{Examples with rank one operators}

As was already mentioned in the introduction, Shargorodsky
\cite{Sha2} constructed $T\in\kk(\ell_1)$ such that (\ref{ps4})
fails for $T$ for one prescribed $\epsilon>0$. We shall demonstrate
that for $X=\ell_1$ and $X=c_0$ one can find a rank 1 operator $T$
for which (\ref{ps4}) fails for any $\epsilon>0$. As usual, we
denote the canonical basis in $c_0$ or $\ell_1$ by
$\{e_n\}_{n\geq0}$.

\begin{example}\label{exc0} Let $T\in L(c_0)$,
$$
Tx=\biggl(\sum_{n=1}^\infty 2^{-n}x_n\biggr)e_0.
$$
Then $T$ has rank $1$, $T^2=0$ and for any $z\in\K$,
$\|T+zI\|=1+|z|$ and the operator $T+zI$ does not attain its norm.
\end{example}

\begin{proof} Obviously, $T$ has rank 1, $\|T\|=1$ and $T^2=0$. Let
$z\in \K$ and $r=|z|$. Since $\|T\|=1$, we have $\|T+zI\|\leq 1+r$.
For $n\in\N$, consider $x_n=(r/z)e_0+e_1+e_2+{\dots}+e_n$. Clearly
$\|x_n\|=1$ and the $e_0$-coefficient of $(T+zI)x_n$ equals
$r+1-2^{-n}$. Hence $\|T+zI\|\geq r+1-2^{-n}$ for any $n\in\N$. Thus
$\|T+zI\|\geq 1+r$. Since the opposite inequality is also true,
$\|T+zI\|=1+r$. It remains to show that $T+zI$ does not attain its
norm. Assume the contrary. Then there exists $x\in c_0$ such that
$\|x\|=1$ and $\|y\|=1+r$, where $y=zx+Tx$. Since $Tx$ is a scalar
multiple of $e_0$, we have $y_n=zx_n$ for $n\in\N$. Hence $|y_n|\leq
r$ for $n\in\N$. Thus $1+r=\|y\|=|y_0|$. Using the definition of $T$
we obtain $y_0=zx_0+\sum\limits_{n=1}^\infty 2^{-n}x_n$. Hence
$$
1+r=|y_0|\leq |z||x_0|+\sum_{n=1}^\infty 2^{-n}|x_n|\leq
r+\sum_{n=1}^\infty 2^{-n}=1+r.
$$
The latter is possible only if $|x_j|=1$ for any $j$ which
contradicts the inclusion $x\in c_0$.
\end{proof}

\begin{example}\label{exl1} Let $T\in L(\ell_1)$,
$$
Tx=\biggl(\sum_{n=1}^\infty (1-2^{-n})x_n\biggr)e_0.
$$
Then $T$ has rank $1$, $T^2=0$ and for any $z\in\C$,
$\|T+zI\|=1+|z|$ and the operator $T+zI$ does not attain its norm.
\end{example}

\begin{proof} Obviously, $T$ has rank 1, $\|T\|=1$ and $T^2=0$. Let
$z\in \K$ and $r=|z|$. For $n\in\N$, we have
$(T+zI)e_n=(1-2^n)e_0+ze_n$. Hence $\|(T+zI)e_n\|=1+r-2^{-n}$. Since
$\|e_n\|=1$, we see that $\|T+zI\|\geq 1+r$. Since the opposite
inequality is also true, $\|T+zI\|=1+r$. It remains to show that
$T+zI$ does not attain its norm. Assume the contrary. Then there
exists $x\in \ell_1$ such that $\|x\|=1$ and $\|y\|=1+r$, where
$y=zx+Tx$. By definition of $T$,
$$
1+r=\|y\|=\biggl|zx_0+\!\sum_{n=1}^\infty
(1-2^{-n})x_n\biggr|\!+\!r\sum_{n=1}^\infty |x_n|\leq
r|x_0|+\sum_{n=1}^\infty (1+r-2^{-n})|x_n|<(1+r)\|x\|=1+r.
$$
The latter inequality is due to the fact that the coefficients $r$
and $1+r-2^{-n}$ in the last sum are strictly less than $1+r$. This
contradiction completes the proof.
\end{proof}

The following Proposition clarifies the situation with the above two
operators and formula (\ref{ps4}).

\begin{proposition}\label{rank1}
Let $T$ be the operator from either Example~$\ref{exc0}$ or
Example~$\ref{exl1}$ in the case $\K=\C$. Then for any $\epsilon>0$,
$\Sigma^0_\epsilon(T)=\sigma_\epsilon(T)\neq\Sigma_\epsilon(T)$.
\end{proposition}

\begin{proof} Since $T^2=0$, $\sigma(T)=\{0\}$ and for any
$z\in\C\setminus\{0\}$, $(T-zI)^{-1}=(-z^{-2})(T+zI)$. Thus
$(T-zI)^{-1}$ attains its norm if and only if so does $T+zI$ and
$\|(T-zI)^{-1}\|=|z|^{-2}\|T+zI\|=|z|^{-1}+|z|^{-2}>|z|^{-1}$. Since
$T+zI$ never attains its norm, we, applying Proposition~\ref{ll0},
see that
$$
\Sigma^0_\epsilon(T)=\sigma_\epsilon(T)=\{z\in\C:|z|^{-1}+|z|^{-2}>\epsilon^{-1}\}=
\{z:|z|<\epsilon+\sqrt{4\epsilon+\epsilon^2}\}\ \ \text{for any}\ \
\epsilon>0.
$$
On the other hand,
$$
\Sigma_\epsilon(T)=\{z\in\C:|z|^{-1}+|z|^{-2}\geq \epsilon^{-1}\}=
\{z:|z|\leq \epsilon+\sqrt{4\epsilon+\epsilon^2}\}\ \ \text{for
any}\ \ \epsilon>0.
$$
Clearly $\Sigma^0_\epsilon(T)\neq\Sigma_\epsilon(T)$ for each
$\epsilon>0$.
\end{proof}

\subsection{Proof of Proposition~\ref{ex}}

Recall that a subset $A$ of a vector space $X$ is called {\it
balanced} if $\lambda x\in A$ whenever $x\in A$, $\lambda\in \K$ and
$|\lambda|\leq1$. A set is called {\it absolutely convex} if it is
convex and balanced. By $\aco(A)$ we denote the absolutely convex
hull of $A$, being the minimal absolutely convex set containing $A$.
Clearly
\begin{equation}\label{ac0}
\aco(A)=\biggl\{\sum_{j=1}^n\lambda_jx_j:n\in\N,\ x_1,\dots,x_n\in
A,\ \lambda_1,\dots,\lambda_n\in \K,\
\sum_{j=1}^n|\lambda_j|\leq1\biggr\}.
\end{equation}
For a subset $A$ of a topological vector space $X$, $\acon(A)$
stands for the closure of $\aco(A)$. We recall two elementary
properties of absolutely convex hulls. The proof of the first one
can be found in virtually any book on topological vector spaces,
see, for instance, \cite{shifer}. The second one is proved in
\cite{bonet}. For a different proof see \cite{stas}.

\begin{lemma}\label{conv1} Let $n\in\N$ and $K_1,\dots,K_n$ be
compact convex subsets of a Hausdorff topological vector space $X$.
Then
$$
\aco\biggl(\bigcup_{j=1}^nK_j\biggr)=\acon\biggl(\bigcup_{j=1}^nK_j\biggr)=
\biggl\{\sum_{j=1}^n\lambda_jx_j:\lambda_j\in\K,\ x_j\in K_j,\
\sum_{j=1}^n|\lambda_j|\leq 1\biggr\}.
$$
Moreover, the above set is compact.
\end{lemma}

\begin{lemma}\label{conv2} Let $\{x_n\}_{n\in\N}$ be a sequence of
elements of a sequentially complete locally convex Hausdorff
topological vector space $X$, converging to $x\in X$ as $n\to
\infty$. Then
$$
\acon(A)=\biggl\{\alpha_0x+\sum_{n=1}^\infty\alpha_nx_n:\alpha\in\ell_1,\
\|\alpha\|_1\leq 1\biggr\},\ \ \text{where}\ \ A=\{x_n:n\in\N\}.
$$
Moreover, $\acon(A)$ is metrizable and compact.
\end{lemma}

From now on in this section, by $\|\cdot\|_2$ we denote the
canonical norm on $\ell_2$. We use the same symbol to denote the
standard Euclidean norm on $\K^2$:
$$
\|(t,s)\|_2=(|t|^2+|s|^2)^{1/2}.
$$
Let also $\{e_n\}_{n\in\N}$ be the canonical orthonormal basis in
$\ell_2$. For $x\in\ell_2$ we denote
$$
x'=x-x_1e_1-x_2e_2.
$$
That is, $x'$ is the orthogonal projection of $x$ onto the closed
linear span of the vectors $e_3,e_4,\dots$. Fix a sequence
$\{q_n\}_{n\in\N}$ of positive numbers such that
$1/2<q_n<1/\sqrt{2}$ for each $n\in\N$ and
$\lim\limits_{n\to\infty}q_n=1/2$. Consider the set $B\subset
\ell_2$,
\begin{equation}\label{B}
B=\biggl\{x+\sum_{n=1}^\infty
(\alpha_n(e_2+e_{n+2})+\beta_nq_n(e_1+e_2+e_{n+2})):\|x'\|_2+\|(x_1,x_2)\|_2+\|\alpha\|_1+\|\beta\|_1
\leq 1\biggr\},
\end{equation}
where $x\in \ell_2$, $\alpha,\beta\in\ell_1$ and $\|\cdot\|_1$ is
the canonical norm in $\ell_1$. Obviously, $B$ is absolutely convex.
Taking into account that $\|e_2+e_{n+2}\|_2=\sqrt{2}$ and
$\|q_n(e_1+e_2+e_{n+2})\|_2=q_n\sqrt{3}\leq\sqrt{3/2}$, we see that
\begin{equation}\label{B1}
\|u\|_2\leq\sqrt{2}\ \ \ \text{for any}\ \ \ u\in B.
\end{equation}
Now if $\|u\|_2\leq 1/2$, then $\|u'\|_2^2+|u_1|^2+|u_2|^2\leq1/4$.
An elementary application of the Cauchy inequality gives
$\|u'\|_2+\|(u_1,u_2)\|_2\leq 1$. Taking $\alpha=\beta=0$ and $x=u$,
we see then that $u\in B$. Thus
\begin{equation}\label{B2}
u\in B\ \ \ \text{if}\ \ \ \|u\|_2\leq1/2.
\end{equation}

We consider the norm $\|\cdot\|$ on $\ell_2$ being the Minkowski
functional of the set $B$. Formulae (\ref{B1}) and (\ref{B2}) imply
that it is indeed a norm and that it is equivalent to the Hilbert
space norm $\|\cdot \|_2$:
$$
2^{-1/2}\|u\|_2\leq \|u\| \leq2\|u\|_2\ \ \text{for all
$u\in\ell_2$}.
$$
In particular, $\ell_2$ endowed with the norm $\|\cdot\|$ is a
reflexive Banach space. Using the definition of the Minkowski
functional, we have that for $u\in \ell_2$,
\begin{equation} \|u\|=\inf
\biggl\{\|x'\|_2+\|(x_1,x_2)\|_2+\|\alpha\|_1+\|\beta\|_1:
u=x+\sum_{n=1}^\infty
(\alpha_n(e_2+e_{n+2})+\beta_nq_n(e_1+e_2+e_{n+2}))\biggr\}.\label{norm}
\end{equation}

We shall show that $B$ coincides with the closed unit ball with
respect to the norm $\|\cdot\|$. Since $B$ is bounded and absolutely
convex, it suffices to show that $B$ is closed in $\ell_2$. First,
note that the set
\begin{equation}\label{b01}
B_1=\{x\in\ell_2:\|x'\|_2+\|(x_1,x_2)\|_2\leq 1\}
\end{equation}
is weakly
compact and $B_1\subseteq B$. Next, let
$B_2=\acon\{e_2+e_{n+2}:n\in\N\}$. Since the sequence $e_2+e_{n+2}$
converges weakly to $e_2$, Lemma~\ref{conv2} implies that $B_2$ is
weakly compact and
\begin{equation}\label{b02}
B_2=\biggl\{se_2+\sum_{n=1}^\infty
\alpha_n(e_2+e_{n+2}):|s|+\|\alpha\|_1\leq 1\biggr\}.
\end{equation}
It follows that $B_2\subseteq B$. Indeed, for
$u=se_2+\sum\limits_{n=1}^\infty \alpha_n(e_2+e_{n+2})\in B_2$, one
has just to take $x=se_2$ and $\beta=0$ to see that $u\in B$.
Similarly, let $B_3=\acon\{q_n(e_1+e_2+e_{n+2}):n\in\N\}$. Since the
sequence $q_n(e_1+e_2+e_{n+2})$ converges weakly to $(e_1+e_2)/2$,
Lemma~\ref{conv2} implies that $B_3$ is weakly compact and
\begin{equation}\label{b03}
B_3=\biggl\{\frac t2(e_1+e_2)+\sum_{n=1}^\infty
\beta_nq_n(e_1+e_2+e_{n+2}):|t|+\|\beta\|_1\leq 1\biggr\}.
\end{equation}
As above, it is clear that $B_3\subseteq B$. Since $B$ is absolutely
convex, we have
$$
B_0=\aco(B_1\cup B_2\cup B_3)\subseteq B.
$$
By Lemma~\ref{conv1}, $B_0$ is weakly compact and
\begin{equation}\label{b00}
B_0=\{ax+by+cw:x\in B_1,\ y\in B_2,\ w\in B_3,\ |a|+|b|+|c|\leq 1\}.
\end{equation}
From formulae (\ref{b01}--\ref{b00}) and (\ref{B}) it follows that
$B\subseteq B_0$. Hence $B=B_0$ and therefore $B$ is weakly compact.
Thus $B$ is closed in $\ell_2$ which ensures that $B$ is the closed
unit ball for the norm (\ref{norm}). It follows that the infimum in
(\ref{norm}) is always attained and that we can write
\begin{equation} \|u\|=\min
\biggl\{\|x'\|_2+\|(x_1,x_2)\|_2+\|\alpha\|_1+\|\beta\|_1:
u=x+\sum_{n=1}^\infty
(\alpha_n(e_2+e_{n+2})+\beta_nq_n(e_1+e_2+e_{n+2}))\biggr\}.\label{normm}
\end{equation}

\begin{lemma}\label{nor} The norm on $\ell_2$ defined by $(\ref{norm})$
satisfies the following conditions:
\begin{itemize}
\item[\rm (\ref{nor}.1)] $\|e_2+e_{n+2}\|=1$ for any $n\in \N$;
\item[\rm (\ref{nor}.2)] $q_n\|e_1+e_2+e_{n+2}\|=1$ for any $n\in
\N$.
\end{itemize}
\end{lemma}

\begin{proof} Taking $x=0$, $\beta=0$ and $\alpha=e_n$, we see that
$e_2+e_{n+2}\in B$. Hence $\|e_2+e_{n+2}\|\leq 1$. Assume that
$\|e_2+e_{n+2}\|<1$. Then there exist $x\in\ell_2$ and
$\alpha,\beta\in\ell_1$ such that
\begin{equation}\label{uu1}
\|x'\|_2+\|(x_1,x_2)\|_2+\|\alpha\|_1+\|\beta\|_1<1
\end{equation}
and
$$
e_2+e_{n+2}=x+\sum_{k=1}^\infty
(\alpha_k(e_2+e_{k+2})+\beta_kq_k(e_1+e_2+e_{k+2})).
$$
Taking the inner product of both sides of the above equality with
$e_2$, we obtain
$$
1=x_2+\sum_{k=1}^\infty(\alpha_k+q_k\beta_k).
$$
Hence
$$
|x_2|+\sum_{k=1}^\infty (|\alpha_k|+q_k|\beta_k|)\geq1,
$$
which contradicts (\ref{uu1}). This contradiction proves
(\ref{nor}.1).

Taking $x=0$, $\alpha=0$ and $\beta=e_n$, we see that
$q_n(e_1+e_2+e_{n+2})\in B$. Hence $q_n\|e_1+e_2+e_{n+2}\|\leq 1$.
Assume that $q_n\|e_1+e_2+e_{n+2}\|<1$. Then there exist
$x\in\ell_2$ and $\alpha,\beta\in\ell_1$ such that
$$
q_n(e_1+e_2+e_{n+2})=x+\sum_{k=1}^\infty
(\alpha_k(e_2+e_{k+2})+\beta_kq_k(e_1+e_2+e_{k+2})).
$$
and (\ref{uu1}) is satisfied. Taking the inner product of both sides
of the above equality with $e_{n+2}$, $e_1$ and $e_2$ we obtain the
following equality in $\K^3$:
\begin{align}\notag
&\quad
q_n(1-\beta_n)(1,1,1)=(x_{n+2},x_1,x_2)+\alpha_n(1,0,1)+\sum_{k\neq
n}(\alpha_k(0,0,1)+\beta_kq_k(0,1,1))=
\\
&=(x_{n+2},\tau,\sigma)+\alpha_n(1,0,1),\ \ \ \text{where}\ \ \
(\tau,\sigma)=(x_1,x_2)+\sum_{k\neq
n}(\alpha_k(0,1)+\beta_kq_k(1,1)). \label{uu3}
\end{align}
Note that $\|(0,1)\|_2=1$ and $\|q_k(1,1)\|_2\leq 1$ since $q_k\leq
1/\sqrt2$. Using (\ref{uu3}) and the triangle inequality, we obtain
$$
\|(\tau,\sigma)\|_2\leq\|(x_1,x_2)\|_2+\sum_{k\neq
n}(|\alpha_k|+|\beta_k|).
$$
From the last display together with (\ref{uu1}) and (\ref{uu3}), it
follows that
$$
q_n(1-\beta_n)(1,1,1)=((x_{n+2}+\alpha_n)e_{n+2},\tau,\sigma+\alpha_n),\
\ \text{where}\ \
|x_{n+2}|+|\alpha_n|+|\beta_n|+\|(\tau,\sigma)\|_2<1.
$$
Dividing by $1-\beta_n$ and denoting $y=x_{n+2}/(1-\beta_n)$,
$a=\alpha_n/(1-\beta_n)$, $r=\tau/(1-\beta_n)$ and
$p=\sigma/(1-\beta_n)$ we see arrive to the following equality in
$\K^3$:
$$
(q_n,q_n,q_n)=(y+a,r,p+a),\ \ \text{where}\ \
|y|+|a|+\sqrt{|r|^2+|p|^2}<1.
$$
Hence $r=q_n$, $p=y=q_n-a$ and
\begin{equation}\label{fi}
|a|+|q_n-a|+\sqrt{|q_n|^2+|q_n-a|^2}<1.
\end{equation}
On the other hand, $q_n>1/2$ and therefore $|a|+|q_n-a|\geq q_n>1/2$
and $\sqrt{|q_n|^2+|q_n-a|^2}\geq q_n>1/2$. Hence
$|a|+|q_n-a|+\sqrt{|q_n|^2+|q_n-a|^2}>1$ which contradicts
(\ref{fi}). This contradiction completes the proof of (\ref{nor}.2).
\end{proof}

\medskip
{\bf Remark.} \ \ In a similar way one can show that $\|u\|=\|u\|_2$
if either $u_1=u_2=0$ or $u$ belongs to the linear span of $e_1$ and
$e_2$
\medskip

Now we consider the operator $S\in L(\ell_2)$ defined by the formula
$Su=u+u_2e_1$. Clearly $S$ is the sum of the identity operator and a
bounded rank 1 operator $Tu=u_2e_1$. Obviously, $T^2=0$.
Proposition~\ref{ex} will be proved if we verify that $\|S\|=2$ and
$S$ does not attain its norm.

\begin{lemma}\label{oper} For any non-zero $u\in \ell_2$,
$\|Su\|<2\|u\|$.
\end{lemma}

\begin{proof} Let $u\in\ell_2$ be such that $\|u\|=1$. It
suffices to show that $\|Su\|<2$. Since $\|u\|=1$, from
(\ref{normm}) it follows that there are $x\in\ell_2$ and
$\alpha,\beta\in\ell_1$ such that
\begin{align}\label{estim2}
&u=x+\sum_{n=1}^\infty
(\alpha_n(e_2+e_{n+2})+\beta_nq_n(e_1+e_2+e_{n+2})),
\\
\label{estim1} &\|x'\|_2+\|(x_1,x_2)\|_2+\|\alpha\|_1+\|\beta\|_1=
1.
\end{align}
Next, from (\ref{estim2}) and the definition of $S$, we obtain that
$$
Su=x+\tau e_1+\sum_{n=1}^\infty q_n(\beta_n+q_n^{-1}\alpha_n)
(e_1+e_2+e_{n+2}),\ \ \ \text{where}\ \ \ \tau=x_2+\sum_{n=1}^\infty
q_n\beta_n.
$$
Using (\ref{norm}), we see that
$$
\|Su\|\leq \|x'\|_2+\|(x_1+\tau,x_2)\|_2+\sum_{n=1}^\infty
|\beta_n+q_n^{-1}\alpha_n|.
$$
From the definition of $\tau$ it follows that
$$
\|Su\|\leq \|x'\|_2+\|(x_1+x_2,x_2)\|_2+\sum_{n=1}^\infty
((1+q_n)|\beta_n|+q_n^{-1}|\alpha_n|).
$$
Taking into account that the norm of the operator with the matrix
$\begin{pmatrix}1&1\\0&1\end{pmatrix}$ acting on the 2-dimensional
Hilbert space $\K^2$ equals $\Bigl(\frac{3+\sqrt
5}{2}\Bigr)^{1/2}<\frac{5}{3}$, we have $\|(x_1+x_2,x_2)\|_2\leq
\frac53 \|(x_1,x_2)\|_2$. Substituting this into the last display
and taking into account that $1+q_n\leq 1+2^{-1/2}<\frac74$, we
obtain
$$
\|Su\|\leq \|x'\|_2+\frac53\|(x_1,x_2)\|_2+\frac74\|\beta\|_1+
\sum_{n=1}^\infty q_n^{-1}|\alpha_n|.
$$
Since the coefficients in the above display in front of $\|x'\|_2$,
$\|(x_1,x_2)\|_2$, $\|\beta\|_1$ and each $|\alpha_n|$ are all
strictly less than 2, formula (\ref{estim1}) implies that
$\|Su\|<2$.
\end{proof}

Now, observe that $S(e_2+e_{n+2})=e_1+e_2+e_{n+2}$. By
Lemma~\ref{nor}, $\|e_2+e_{n+2}\|=1$ and
$\|e_1+e_2+e_{n+2}\|=q_n^{-1}$. Hence $\|S\|\geq q_n^{-1}$ for any
$n\in\N$. Since $q_n^{-1}\to 2$ as $n\to\infty$, we have $\|S\|\geq
2$. Thus from Lemma~\ref{oper} it follows that $\|S\|=2$ and $S$
does not attain its norm. This completes the proof of
Proposition~\ref{ex}.

\section{Concluding remarks}

{\bf 1. }\ A more general approach to study the class $\cla$ is to
consider the following property.

{\bf Definition 3.} \ We say that a Banach space $X$ is {\it tame}
if for any $y\in X$, $x\in X\setminus\{0\}$ and any sequence
$\{u_n\}_{n\in\N}$ in $X$ weakly convergent to zero,
\begin{equation}\label{tame}
\slim_{n\to\infty}\frac{\|y+u_n\|}{\|x+u_n\|}\leq\max\biggl\{1,\frac{\|y\|}{\|x\|}\biggr\}.
\end{equation}

It is easy to see that $p$-spaces for $1<p<\infty$ are tame.

\begin{proposition}\label{tam1}
Let $X$ be a reflexive Banach space such that either $X$ or $X^*$ is
tame. Then $X\in\cla$.
\end{proposition}

\begin{proof} According to Corollary~\ref{du3}, it is sufficient to
consider the case, when $X$ is tame. Let $T\in\kk(X)$ and
$\|I+T\|=c>1$. Since $X$ is reflexive, we can pick a sequence
$\{x_n\}_{n\in\N}$ of elements of $X$ such that $\|x_n\|\to 1$,
$\|(I+T)x_n\|\to c$ and $\{x_n\}$ is weakly convergent to $x\in X$.
By Lemma~\ref{xnz}, $x\neq 0$. Since the sequence $u_n=x_n-x$ is
weakly convergent to $0$ and $T$ is compact, $Tx_n$ is
norm-convergent to $Tx$. Hence $\|x+Tx+u_n\|=\|x_n+Tx\|\to c$. Since
$X$ is tame, we have
$$
c=\lim_{n\to\infty}\frac{\|x+Tx+u_n\|}{\|x+u_n\|}\leq\max\biggl\{
1,\frac{\|x+Tx\|}{\|x\|}\biggr\}.
$$
The inequality $c>1$ and the last display imply that $\|x+Tx\|\geq
c\|x\|$. Taking into account that $c=\|I+T\|$, we see that $I+T$
attains its norm on $x/\|x\|$.
\end{proof}

Unfortunately, it seems there are no known examples of tame Banach
spaces which are not $p$-spaces. This naturally leads to the problem
of characterizing the tame spaces.

\bigskip

\noindent {\bf 2. }\ Analyzing the proof of Theorem~\ref{lp2}, one
can easily see that if $1<p<\infty$, $X$ is a $p$-space and $T\in
\kk(X)$ is such that $\|I+K\|>1$, then whenever $\{x_n\}_{n\in\N}$
is a sequence of elements of $X$ weakly converging to $x\in X$ and
satisfying $\|x_n\|\to 1$, $\|(I+T)x_n\|\to \|I+T\|$ as
$n\to\infty$, then $\{x_n\}$ is norm convergent to $x$ and $I+T$
attains its norm on $x$.

\bigskip

\noindent{\bf 3. }\ In a recent paper \cite{SHSH} Shargorodsky and
the author constructed a strictly convex reflexive Banach space $X$
and $S\in L(X)$ such that for some $\epsilon>0$, the level set
$\Sigma_\epsilon(S)\setminus\sigma_\epsilon(S)$ has non-empty
interior. The space $X$ constructed in \cite{SHSH} is an
$\ell_2$-direct sum of a countable family of finite dimensional
Banach spaces. Thus by Theorem~\ref{lp1}, (\ref{ps4}) holds for any
$T\in L(X)$. This observation shows that there is no relation
between validity of (\ref{ps4}) and meagreness of the level sets of
the norm of the resolvent.

\bigskip

\noindent{\bf 4. }\ Let $1<p<\infty$. As it follows from
Proposition~\ref{uvuv} and Theorem~\ref{lp2}, any subspace of an
$\ell_p$-direct sum $X$ of a family of finite dimensional Banach
spaces belongs to $\cla$. Applying Corollary~\ref{du3}, one can
easily see that same holds true for quotients of $X$ as well.
Indeed, $X^*$ is naturally isometrically isomorphic to an
$\ell_{p'}$-direct sum of a family of finite dimensional Banach
spaces, where $\frac1p+\frac1{p'}=1$. Moreover, for any closed
linear subspace $Y$ of $X$, $(X/Y)^*$ is naturally isometrically
isomorphic to a subspace of $X^*$.

\bigskip

\noindent {\bf 5. }\ It would be interesting to figure out which
classical Banach spaces do belong to the class $\cla$. A good
starting point would be to address the spaces $L_p[0,1]$ for
$1<p<\infty$. There is a strong indication against their membership
in $\cla$ for $p\neq 2$. Namely, it is easy to show that $L_p[0,1]$
for $p\neq 2$ is not tame.

\bigskip

\noindent {\bf 6. }\ The compactness condition in
Propositions~\ref{ll} and~\ref{ll0} can be replaced by the weaker
condition of $T$ being strictly singular. The proofs work without
any changes.

\small\rm

\vskip1truecm

\scshape

\noindent Stanislav Shkarin

\noindent Queen's University Belfast

\noindent Department of Pure Mathematics

\noindent University road, BT7 1NN \ Belfast, UK

\noindent E-mail address: \qquad {\tt s.shkarin@qub.ac.uk}

\end{document}